\documentclass[12pt, reqno]{amsart}
\usepackage{amsmath, amsthm, amscd, amsfonts, amssymb, graphicx, color}
\usepackage[bookmarksnumbered, colorlinks, plainpages]{hyperref}

\input{mathrsfs.sty}

\newtheorem{theorem}{Theorem}[section]
\newtheorem{lemma}[theorem]{Lemma}
\newtheorem{proposition}[theorem]{Proposition}
\newtheorem{corollary}[theorem]{Corollary}
\theoremstyle{definition}

\theoremstyle{remark}

\numberwithin{equation}{section}
\begin{document}

\title [Some operator Bellman type inequalities ]{Some  operator Bellman type inequalities }

\author[M. Bakherad, A. Morassaei]{Mojtaba Bakherad$^1$ and Ali Morassaei$^2$}

\address{$^1$Department of Mathematics, Faculty of Mathematics, University of Sistan and Baluchestan, Zahedan, Iran.}
\email{mojtaba.bakherad@yahoo.com; bakherad@member.ams.org}
\address{$^2$Department of Mathematics, Faculty of Sciences, University of Zanjan, University Blvd., Zanjan 45371-38791, Iran}
\email{morassaei@znu.ac.ir, morassaei@chmail.ir}

\begin{abstract}
In this paper, we employ the Mond--Pe\v{c}ari\'c method to establish some reverses of the operator Bellman inequality under certain conditions. In particular, we show
\begin{equation*}
\delta I_{\mathscr K}+\sum_{j=1}^n\omega_j\Phi_j\left((I_{\mathscr H}-A_j)^{p}\right)\ge \left(\sum_{j=1}^n\omega_j\Phi_j(I_{\mathscr H}-A_j)\right)^{p} \,,
\end{equation*}
where $A_j\,\,(1\leq j\leq n)$ are self-adjoint  contraction operators with $0\leq mI_{\mathscr H}\le A_j \le MI_{\mathscr H}$,  $\Phi_j$ are unital positive linear maps on ${\mathbb B}({\mathscr H})$, $\omega_j\in\mathbb R_+ \,\,(1\leq j\leq n)$, $0 < p < 1$ and
 $\delta=(1-p)\left(\frac{1}{p}\frac{(1-m)^p-(1-M)^p}{M-m}\right)^{\frac{p}{p-1}}+\frac{(1-M)(1-m)^p-(1-m)(1-M)^p}{M-m}$ .
 We also present some refinements of the operator Bellman inequality. \end{abstract} \maketitle

\section{Introduction}
Let ${\mathbb B}({\mathscr H})$ denote the $C^*$-algebra of all
bounded linear operators on a complex Hilbert space ${\mathscr H}$ with the identity $I_{\mathscr H}$. In the case when ${\rm dim}{\mathscr{H}}=n$, we identify $\mathbb{B}(\mathscr{H})$ with the matrix algebra $\mathcal{M}_n(\mathbb{C})$ of all
$n\times n$ matrices with entries in the complex field.
An operator $A\in{\mathbb B}({\mathscr H})$ is called positive
if $\langle Ax,x\rangle\geq0$ for all $x\in{\mathscr H }$  and in this case we write $A\geq0$. We write $A>0$ if $A$ is
a positive invertible operator. The set of all positive invertible operators  is denoted by ${\mathbb B}({\mathscr H})_+$. For
self-adjoint operators $A, B\in{\mathbb B}({\mathscr H})$, we say
$A \leq B$ if $B-A\geq0$. Also, an operator $A\in {\mathbb B}({\mathscr H})$ is said to be contraction, if $A^*A \leq I_{\mathscr H}$. The Gelfand map $f(t)\mapsto f(A)$ is an
isometrical $*$-isomorphism between the $C^*$-algebra
$C({\rm sp}(A))$ of continuous functions on the spectrum ${\rm sp}(A)$
of a self-adjoint operator $A$ and the $C^*$-algebra generated by $A$ and $I_{\mathscr H}$. If $f, g\in C({\rm sp}(A))$, then
$f(t)\geq g(t)\,\,(t\in{\rm sp}(A))$ implies that $f(A)\geq g(A)$.

Let $f$ be a continuous real valued function defined on an interval $J$. It is called operator monotone if
$A\leq B$ implies $f(A)\leq f(B)$ for all self-adjoint operators $A, B\in {\mathbb B}({\mathscr H})$ with spectra in $J$; see \cite{FKN}  and references therein for some recent results. It said to be operator concave if $\lambda f(A)+(1-\lambda)f(B)\leq f(\lambda A+(1-\lambda)B)$ for all self-adjoint operators $A, B\in {\mathbb B}({\mathscr H})$ with spectra in $J$ and all $\lambda\in [0,1]$. Every nonnegative continuous function $f$ is operator
monotone on $[0,+\infty)$ if and only if $f$ is operator concave on $[0,+\infty)$; see \cite[Theorem 8.1]{abc}.
A map $\Phi:\mathbb{B}(\mathscr{H}) \longrightarrow\mathbb{B}(\mathscr{K})$ is called
positive if $\Phi(A)\geq0$ whenever $A\geq0$, where $\mathscr K$ is a complex Hilbert space and is said to be unital if $\Phi(I_\mathscr{H})=I_\mathscr{K}$.
We denote by $\mathbf{P}[\mathbb{B}(\mathscr{H}),\mathbb{B}(\mathscr{K})]$ the set of all positive linear maps $\Phi:\mathbb{B}(\mathscr{H})\to\mathbb{B}(\mathscr{K})$ and by $\mathbf{P}_N[\mathbb{B}(\mathscr{H}),\mathbb{B}(\mathscr{K})]$ the set of all unital positive linear maps.

The axiomatic theory for operator means of positive invertible operators have been developed by Kubo and Ando \cite{ando}. A binary operation $\sigma$ on ${\mathbb B}({\mathscr H})_+$ is called a connection, if the following conditions are satisfied:
\begin{itemize}
\item [(i)] $A\leq C$ and $B\leq D $ imply
$A\sigma B\leq C\sigma D$;
\item [(ii)] $A_n\downarrow A$ and $ B_n\downarrow B$ imply
$A_n\sigma B_n\downarrow A\sigma B$, where $A_n\downarrow A$ means that $A_1\geq A_2\geq \cdots$ and $A_n\rightarrow A$ as $n\rightarrow\infty$ in the strong operator topology;
\item [(iii)] $T^*(A\sigma B)T\leq (T^*AT)\sigma (T^*BT)\,\,(T\in{\mathbb B}({\mathscr H}))$.
\end{itemize}
There exists an affine order isomorphism between the class of connections and the class of positive operator monotone functions $f$ defined on $(0,\infty)$  via
$f(t)I_{\mathscr H}=I_{\mathscr H}\sigma_f(tI_{\mathscr H})\hspace{.1cm}(t>0)$. In addition, $A\sigma_f B=A^{1\over 2}f(A^{-\frac{1}{2}}BA^{-\frac{1}{2}})A^{1\over2}$ for all $A, B\in{\mathbb B}({\mathscr H})_+$. The operator monotone function $f$ is called the representing function of $\sigma_f$.  A connection $\sigma_f$ is a mean if it is normalized, i.e. $I_{\mathscr H}\sigma_f I_{\mathscr H}=I_{\mathscr H}.$ The function $f_{\nabla_\mu}(t)=(1-\mu)+\mu t$ and $f_{\sharp_\mu}(t)=t^\mu$ on $(0,\infty)$ for $\mu\in(0,1)$ give the operator weighted arithmetic mean $A\nabla_\mu B=(1-\mu)A+\mu B$ and the operator weighted geometric mean $A\sharp_\mu
B=A^{\frac{1}{2}}\left(A^{-\frac{1}{2}}BA^{-\frac{1}{2}}\right)^{\mu}A^{\frac{1}{2}}$, respectively. The case $\mu=1/2$, the operator weighted geometric mean gives rise to the so--called geometric mean $A\sharp B$.

Bellman \cite{bell} proved that if $p$ is a positive integer and  $a, b, a_j, b_j\,\,(1\leq j\leq n)$  are positive real numbers  such that $\sum_{j=1}^na_j^p \leq a^p$ and $\sum_{j=1}^nb_j^p \leq b^p$, then
\begin{align*}
\left(a^p-\sum_{j=1}^na_j^p\right)^{1/p}+ \left(b^p-\sum_{j=1}^nb_j^p\right)^{1/p} \leq \left((a+b)^p-\sum_{j=1}^n(a_j+b_j)^p\right)^{1/p}\,.
\end{align*}
A "multiplicative" analogue of this inequality is due to J. Acz\'{e}l. In 1956, Acz\'{e}l \cite{Ac} proved that if $a_j, b_j~ (1\le j\le n)$ are positive real numbers such that $a_1^2-\sum_{j=2}^na_j^2>0$ or $b_1^2-\sum_{j=2}^nb_j^2>0$, then
$$
\left(a_1^2-\sum_{j=2}^na_j^2\right)\left(b_1^2-\sum_{j=2}^nb_j^2\right)\leq \left(a_1b_1-\sum_{j=2}^na_jb_j\right)^2\,.
$$
Popoviciu \cite{Po} extended Acz\'{e}l's inequality by showing that
$$
\left(a_1^p-\sum_{j=2}^na_j^p\right)\left(b_1^p-\sum_{j=2}^nb_j^p\right)\leq \left(a_1b_1-\sum_{j=2}^na_jb_j\right)^p\,,
$$
where $p\ge 1$ and $a_1^p-\sum_{j=2}^na_j^p>0$ or $b_1^p-\sum_{j=2}^nb_j^p>0$.

During the last decades several generalizations, refinements and applications of the Bellman inequality in various settings have been given and some results related to integral inequalities are presented; see \cite{pop, MMM1, MMM, Moslehian, mor5} and references therein.

In \cite{MMM} the authors showed an operator Bellman inequality as follows:
\begin{align*}
\Phi\left((I_\mathscr{H}-A)^{p}\nabla_\lambda(I_\mathscr{H}-B)^{p}\right) \leq \left(\Phi(I_\mathscr{H}-A\nabla_\lambda B)\right)^{p}\,,
\end{align*}
whenever $A, B$ are positive contraction operators, $\Phi$ is a unital positive linear map on $\mathbb{B}(\mathscr{H})$ and $0 < p < 1$.
They also \cite{MMM1} showed the following generalization of the Bellman operator inequality
\begin{align}\label{mor2}
\left(I_\mathscr{H}-\sum_{j=1}^nA_j\right)\sigma_{f^p}\left(I_\mathscr{H}-\sum_{j=1}^nB_j\right)
\leq\left(I_\mathscr{H}-\left(\sum_{j=1}^nA_j\sigma_fB_j\right)\right)^{p},
\end{align}
where $A_j,B_j\,\,(1\leq j\leq n)$ are positive operators such that $\sum_{j=1}^nA_j \leq I_\mathscr{H}$, $\sum_{j=1}^nB_j \leq I_\mathscr{H}$, $\sigma_f$ is a mean with the representing function $f$ and $0 < p < 1$.

In this paper, we use the Mond--Pe\v{c}ari\'c method to present some reverses of the operator Bellman inequality under some mild conditions. We also show some refinements of \eqref{mor2}.
\section{Some reverses of the Bellman type operator inequality}
The operator Choi-Davis-Jensen inequality says that if $f$ is an operator concave function on an interval $J$ and $\Phi\in\mathbf{P}_N[\mathbb{B}(\mathscr{H}),\mathbb{B}(\mathscr{K})]$, then $f(\Phi(A))\geq \Phi(f(A))$ for all self-adjoint operators $A$ with spectrum in $J$. The Mond--Pe\v{c}ari\'c method \cite[Chapter 2]{abc} present that if $f$ is a strictly  concave differentiable function on an interval $[m,M]$ with $m<M$ and  $\Phi\in\mathbf{P}_N[\mathbb{B}(\mathscr{H}),\mathbb{B}(\mathscr{K})]$,
{\footnotesize\begin{eqnarray}\label{mpm2}
\mu_f=\frac{f(M)-f(m)}{M-m}\,, \nu_f=\frac{Mf(m)-mf(M)}{M-m}\,\,{\rm~and~}\,\,\gamma_f=\max\left\{\frac{f(t)}{\mu_f t+\nu_f}: m\leq t\leq M\right\}\,,
\end{eqnarray}}
then
\begin{eqnarray}\label{mond1}
\gamma_f \Phi(f(A)) \geq f(\Phi(A)).
\end{eqnarray}
In inequality \eqref{mond1}, if we put $\Phi(X):=\Psi(A)^{-1/2}\Psi(A^{1/2}XA^{1/2})\Psi(A)^{-1/2}$, where $\Psi$ is an arbitrary unital positive linear map and take $f$ to be the representing function of an operator mean $\sigma_f$, then we reach the inequality
\begin{eqnarray}\label{mond2}
\left(\displaystyle{\max_{m\leq t\leq M}} \frac{f(t)}{\mu_f t+\nu_f}\right) \Psi(A\sigma_f B) \geq \Psi(A)\sigma_f \Psi(B)
\end{eqnarray}
whenever $0< mA\leq B\leq MA$.\\
Finally, if we take $\Psi$ in \eqref{mond2} to be the positive unital linear map defined on the diagonal blocks of operators by $\Psi({\rm diag}(A_1, \cdots, A_n))=\frac{1}{n}\sum_{j=1}^nA_j$, then
\begin{eqnarray}\label{main}
\gamma_f \sum_{j=1}^n A_j\sigma_f B_j \geq \left(\sum_{j=1}^nA_j\right)\sigma_f \left(\sum_{j=1}^nB_j\right),
\end{eqnarray}
where $\gamma_f$ is given by \eqref{mpm2} and $0< mA_j\leq B_j\leq MA_j\,\,(1\leq j\leq n)$.\\

In the following theorem we show a  reversed operator Bellman type inequality.
\begin{theorem}\label{such1}
Suppose that $0<m A_j \leq B_j \leq M A_j\,\,(1\leq j\leq n)$ and $0<m\left(I_{\mathscr H}-\gamma_f\sum_{j=1}^nA_j\right)\leq I_{\mathscr H}-\gamma_f\sum_{j=1}^nB_j\leq M\left(I_{\mathscr H}-\gamma_f\sum_{j=1}^nA_j\right)$ for some positive real numbers $m, M$  such that $m<1<M$, $\gamma_f$ is given by \eqref{mpm2}, $\sigma_f$ is an operator mean with the representing function $f$ and $p\in[0,1]$. Then
{\footnotesize\begin{eqnarray}\label{harry}
\gamma_f^p\left(\left(I_{\mathscr H}- \sum_{j=1}^nA_j\right)\sigma_{f}\left(I_{\mathscr H}-\sum_{j=1}^nB_j\right)\right)^p
\geq\left(I_{\mathscr H}-\gamma_f\left(\sum_{j=1}^nA_j\sigma_f B_j\right)\right)^p.
\end{eqnarray}}
\end{theorem}
\begin{proof}
By using \eqref{main} we have
{\footnotesize\begin{align*}
\gamma_f \sum_{j=1}^{n+1} X_j\sigma_f Y_j \geq \left(\sum_{j=1}^{n+1}X_j\right)\sigma_f \left(\sum_{j=1}^{n+1}Y_j\right),
\end{align*}}
where $0<m X_j \leq Y_j \leq M X_j\,\,(1\leq j\leq n+1)$. If we take
$X_j=A_j, Y_j=B_j\,\,(1\leq j\leq n)$, $X_{n+1}=I_{\mathscr H}-\sum_{j=1}^nA_j\geq0$ and $Y_{n+1}=I_{\mathscr H}-\sum_{j=1}^nB_j\geq0$, then
{\footnotesize\begin{align*}
\gamma_f \left[\left(\sum_{j=1}^{n}A_j\right)\sigma_f \left(\sum_{j=1}^{n}B_j\right)
+\left(I_{\mathscr H}-\sum_{j=1}^nA_j\right)\sigma_f\left(I_{\mathscr H}-\sum_{j=1}^nB_j\right)\right]\geq I_{\mathscr H}\sigma_f I_{\mathscr H}=I_{\mathscr H},
\end{align*}}
whence
{\footnotesize\begin{align}\label{mond-mor2}
\left(I_{\mathscr H}-\sum_{j=1}^nA_j\right)\sigma_f\left(I_{\mathscr H}-\sum_{j=1}^nB_j\right)\geq\frac{1}{\gamma_f}I_{\mathscr H}-\left(\sum_{j=1}^{n}A_j\right)\sigma_f \left(\sum_{j=1}^{n}B_j\right).
\end{align}}
It follows from inequality \eqref{mond-mor2} and the L\"owner-Heinz inequality \cite[Theorem 1.8]{abc} that
{\footnotesize\begin{align*}
\left[\left(I_{\mathscr H}-\sum_{j=1}^nA_j\right)\sigma_f\left(I_{\mathscr H}-\sum_{j=1}^nB_j\right)\right]^p
&\geq\left[\frac{1}{\gamma_f}I_{\mathscr H}-\left(\sum_{j=1}^{n}A_j\right)\sigma_f \left(\sum_{j=1}^{n}B_j\right)\right]^p.
\end{align*}}
\end{proof}

\begin{lemma}\label{reverses2}
Suppose that $C, X\in\mathbb{B}(\mathscr{H})$ such that $C$ is a contraction, $0<m I_{\mathscr H}\leq X \leq M I_{\mathscr H}$, $f$ is a concave and operator monotone function on $[m,M]$ and $\gamma_f$ is given by \eqref{mpm2}. Then
\begin{align}\label{bomb}
\gamma_f\left[C^*f(X)C+f(m)(I_{\mathscr H}-C^*C)\right]\geq f(C^*XC).
\end{align}
\end{lemma}
\begin{proof}
Let $D=(I_{\mathscr H}-C^*C)^\frac{1}{2}$. Consider the positive unital linear map $\Phi\left(\left[\begin{array}{cc}
         X &  0\\0 & Y
 \end{array}\right]\right)=C^*XC+D^*YD\,\,(X,Y\in\mathbb{B}(\mathscr{H}))$. Using inequality \eqref{mond1} and the operator monotonicity of $f$ we have
\begin{align*}
f(C^*XC)&\leq f(C^*XC+D^*mD)\\&=f\left(\Phi\left(\left[\begin{array}{cc}
         X &  0\\0 & m
 \end{array}\right]\right)\right)\\&\leq\gamma_f\left(\Phi\left(\left[\begin{array}{cc}
         f(X) &  0\\0 & f(m)
 \end{array}\right]\right)\right)\,\,(\textrm{by \eqref{mond1}})\\&=\gamma_f\left[C^*f(X)C+D^*f(m)D\right],
\end{align*}
whenever $0<m I_{\mathscr H}\leq X \leq M I_{\mathscr H}$.
\end{proof}
\begin{lemma}\label{reverses2}
Let $0<m A \leq B \leq M A$ with $A$ contraction. Let $\sigma_f$ be an operator mean with the representing function $f$ and $h$ be an operator monotone function on $[0,+\infty)$. Then
\begin{eqnarray}\label{bos12}
\gamma_h\big[h(f(m))\big(I_{\mathscr H}-A\big)+\left(A\sigma_{hof}B\right)\big]\geq h\left(A\sigma_fB\right),
\end{eqnarray}
where  $\mu_h=\frac{h(f(M))-h(f(m))}{f(M)-f(m)}, \nu_h=\frac{f(M)h(f(m))-f(m)h(f(M))}{f(M)-f(m)}$ and $\gamma_h=\displaystyle{\max_{f(m)\leq t\leq f(M)}} \frac{h(t)}{\mu_h t+\nu_h}$.
\end{lemma}
\begin{proof}
It follows from $f(m)\leq f\left(A^{-1/2}BA^{-1/2}\right)\leq f(M)$ and inequality \eqref{bomb} that
\begin{align*}
h\left(A\sigma_fB\right)&= h\left(A^{1/2}f\left(A^{-1/2}BA^{-1/2}\right)A^{1/2}\right)\\&
\leq\gamma_h\left(A^{1/2}h\left(f\left(A^{-1/2}BA^{-1/2}\right)\right)A^{1/2}+\left(I_{\mathscr H}-A\right)^{1/2}h(f(m))\left(I_{\mathscr H}-A\right)^{1/2}\right)\\&
\qquad\qquad\qquad\qquad\qquad\qquad\qquad\qquad\qquad\qquad\qquad(\textrm{by \eqref{bomb}})\\&=
\gamma_h\left[\left(A\sigma_{hof}B\right)+h(f(m))\left(I_{\mathscr H}-A\right)\right].
\end{align*}
\end{proof}
Applying the operator monotone function $f(t)=(1-\lambda)+\lambda t\,\,(\lambda\in[0,1])$ in Theorem \ref{such1},  due to $$\gamma_f=\displaystyle{\max_{m\leq t\leq M}} \frac{f(t)}{\mu_f t+\nu_f}=\displaystyle{\max_{m\leq t\leq M}} \frac{(1-\lambda)+\lambda t}{(1-\lambda) +\lambda t}=1$$
and using Lemma \ref{reverses2} for the special case $h(t)=t^p\,\,(p\in[0,1])$, due to  $$\gamma_h=\displaystyle{\max_{f(m)\leq t\leq f(M)}} \frac{t^p}{\mu_h t+\nu_h}=\frac{p^p (f(M)-f(m)) (f(M)f(m)^p-f(m)f(M)^p)^{p-1}}{(1-p)^{p-1}(f(M)^p-f(m)^p)^{p}}$$
we have the following result; see \cite[p. 77]{abc}.
\begin{corollary}[A reverse operator Bellman inequality]
Let $0<m A_j \leq B_j \leq M A_j\,\,(1\leq j\leq n)$ and $0<m\left(I_{\mathscr H}-\sum_{j=1}^nA_j\right)\leq I_{\mathscr H}-\sum_{j=1}^nB_j\leq M\left(I_{\mathscr H}-\sum_{j=1}^nA_j\right)$ for some positive real numbers $m, M$  such that $m<1<M$ and $p\in[0,1]$. Then
{\footnotesize\begin{eqnarray*}
\delta\left(f(m)^p\left(\sum_{j=1}^nA_j\right)+\left(I_{\mathscr H}- \sum_{j=1}^nA_j\right)\sigma_{((1-\lambda)+\lambda t)^p}\left(I_{\mathscr H}-\sum_{j=1}^nB_j\right)\right)
\geq\left(I_{\mathscr H}-\left(\sum_{j=1}^nA_j\sigma_{f} B_j\right)\right)^p,
\end{eqnarray*}}
where  $\delta=\frac{p^p (f(M)-f(m)) (f(M)f(m)^p-f(m)f(M)^p)^{p-1}}{(1-p)^{p-1}(f(M)^p-f(m)^p)^{p}}$ and $f(t)=(1-\lambda)+\lambda t\,\,(\lambda\in[0,1])$.
\end{corollary}
\begin{proof}
We have
{\footnotesize\begin{align*}
\delta&\left(f(m)^p\left(\sum_{j=1}^nA_j\right)+\left(I_{\mathscr H}- \sum_{j=1}^nA_j\right)\sigma_{((1-\lambda)+\lambda t)^p}\left(I_{\mathscr H}-\sum_{j=1}^nB_j\right)\right)
\\&\geq\left(\left(I_{\mathscr H}- \sum_{j=1}^nA_j\right)\nabla_\lambda\left(I_{\mathscr H}-\sum_{j=1}^nB_j\right)\right)^p\,\,(\textrm{by \eqref{bos12}})
\\&\geq\left(I_{\mathscr H}-\left(\sum_{j=1}^nA_j\nabla_\lambda B_j\right)\right)^p\,\,(\textrm{by \eqref{harry}}).
\end{align*}}
\end{proof}
In \cite{rma}, the authors showed another way to find a reverse Choi-Davis-Jensen inequality.
If $f$ is a strictly concave differentiable function on an interval $[m,M]$ with $m<M$ and  $\Phi$ is a unital positive linear map, then
\begin{eqnarray}\label{mond1o1}
\beta_f I_{\mathscr H}+ \Phi(f(A)) \geq f(\Phi(A)),
\end{eqnarray}
where $A \in {\mathbb B}({\mathscr H})$ is a self-adjoint operator with spectrum in $[m,M]$ and $\beta_f=\max_{m\leq t\leq M}\left\{f(t)-\mu_ft-\nu_f\right\}$.

In inequality \eqref{mond1o1}, if we put $\Psi(X):=\Psi(A)^{-1/2}\Psi(A^{1/2}XA^{1/2})\Psi(A)^{-1/2}$, where $\Psi$ is an arbitrary unital positive linear map and take $f$ to be the representing function of an operator mean $\sigma_f$, then we reach the inequality
\begin{eqnarray}\label{kour}
\beta_f\Psi(X)+\Psi(X\sigma_f Y)\geq\Psi(X)\sigma_f \Psi(Y),
\end{eqnarray}
where $0<m X \leq Y \leq M X$, $\sigma_f$ is an operator mean with representing function $f$ and $\beta_f=\max_{m\leq t\leq M}\left\{f(t)-\mu_ft-\nu_f\right\}$ that is the unique solution of the equation $f'(t)=\mu_f$, whenever $\mu_f=\frac{f(M)-f(m)}{M-m}$ and $\nu_f=\frac{Mf(m)-mf(M)}{M-m}$.\\
Applying \eqref{kour} to the positive unital linear map defined on the diagonal blocks of operators by $\Psi({\rm diag}(X_1, \cdots, X_{n+1}))=\frac{1}{n}\sum_{j=1}^{n+1}X_j$,  we get
\begin{eqnarray}\label{reverse234}
\beta_f\sum_{j=1}^{n+1}X_j+\sum_{j=1}^{n+1}Y_j\sigma_f X_j\geq\left(\sum_{j=1}^{n+1}X_j\right)\sigma_f \left(\sum_{j=1}^{n+1}Y_j\right),
\end{eqnarray}
where $0<m X_j \leq Y_j \leq M X_j\,\,(1\leq j\leq n+1)$, $\sigma_f$ is an operator mean with representing function $f$ and $\beta_f=\max_{m\leq t\leq M}\left\{f(t)-\mu_ft-\nu_f\right\}.$

Now, we have the next result.
\begin{proposition}
Suppose that  $0<m A_j \leq B_j \leq M A_j\,\,(1\leq j\leq n)$ and $0<m\left(I_{\mathscr H}-\sum_{j=1}^nA_j\right)\leq I_{\mathscr H}-\sum_{j=1}^nB_j\leq M\left(I_{\mathscr H}-\sum_{j=1}^nA_j\right)$ for some positive real numbers $m, M$  such that $m<1<M$ and  $\sigma_f$ is an operator mean with representing function $f$. Then
\begin{align*}
\left(\beta_f+\left(I_{\mathscr H}-\sum_{j=1}^nA_j\right)\sigma_f\left(I_{\mathscr H}-\sum_{j=1}^nB_j\right)
\right)^p\geq \left(I_{\mathscr H}-\left(\sum_{j=1}^nA_j\sigma_fB_j\right)\right)^p,
\end{align*}
whenever $p\in[0,1]$ and $\beta_f=\max_{m\leq t\leq M}\left\{f(t)-\mu_ft-\nu_f\right\}.$
\end{proposition}
\begin{proof}
If we take
$X_j=A_j, Y_j=B_j\,\,(1\leq j\leq n)$, $X_{n+1}=I_{\mathscr H}-\sum_{j=1}^nA_j$ and $Y_{n+1}=I_{\mathscr H}-\sum_{j=1}^nB_j$ in inequality \eqref{reverse234}, then we get
\begin{align}\label{1eq1}
\beta_f+\left(I_{\mathscr H}-\sum_{j=1}^nA_j\right)\sigma_f\left(I_{\mathscr H}-\sum_{j=1}^nB_j\right)
\geq \left(I_{\mathscr H}-\left(\sum_{j=1}^nA_j\sigma_fB_j\right)\right).
\end{align}
By the operator monotonicity of $h(t)=t^p$ and \eqref{1eq1} we reach the desired inequality.
\end{proof}
\begin{corollary}[A reverse Acz\'{e}l type inequality]
Let $0<m A_j \leq B_j \leq M A_j\,\,(1\leq j\leq n)$, $0<m\left(I_{\mathscr H}-\sum_{j=1}^nA_j\right)\leq I_{\mathscr H}-\sum_{j=1}^nB_j\leq M\left(I_{\mathscr H}-\sum_{j=1}^nA_j\right)$ for some positive real numbers $m, M$  such that $m<1<M$. Then
\begin{align*}
\left(\zeta+\left(I_{\mathscr H}-\sum_{j=1}^nA_j\right)\sharp_\lambda\left(I_{\mathscr H}-\sum_{j=1}^nB_j\right)\right)^p
\geq \left(I_{\mathscr H}-\left(\sum_{j=1}^nA_j\sharp_\lambda B_j\right)\right)^p,
\end{align*}
where $p,\lambda\in[0,1]$ and $\zeta=(1-p)\left(\frac{M^p-m^p}{p(M-m)}\right)^{\frac{p}{p-1}}-\frac{Mm^p-mM^p}{M-m}$.
\end{corollary}
Let $A_j\in \mathbb B(\mathscr H)~(1 \leq j \leq n)$ be self-adjoint operators with ${\rm sp}(A_j)\subseteq[m,M]$ for some scalars $m<M$, $\Phi_j$ be unital positive linear maps on $\mathbb B(\mathscr H)$, $\omega_1,\cdots,\omega_n\in\mathbb R_+$ be any finite number of positive real numbers such that $\sum_{j=1}^n\omega_j=1$ and $f$ be a strictly concave differentiable  function. If we take the positive unital linear map $\Phi({\rm diag}(A_1, \cdots, A_{n}))=\sum_{j=1}^{n}\omega_j\Phi_j(A_j)$,  in inequality \eqref{mond1o1},  then
\begin{equation}\label{JR}
\beta_fI+\sum_{j=1}^n\omega_j\Phi_j\big(f(A_j)\big)\ge f\left(\sum_{j=1}^n\omega_j\Phi_j(A_j)\right),
\end{equation}
where $\beta_f=\max_{m\leq t\leq M}\left\{f(t)-\mu_ft-\nu_f\right\}$; see also \cite[Corollary 2.16]{abc}.

We can generalize the operator Bellman inequality in Corollary 2.2 of \cite{MMM} as follows
\begin{equation}\label{BELL}
\left(\Phi\left(I_{\mathscr H}-\sum_{j=1}^n\omega_j A_j\right)\right)^p \geq \Phi\left(\sum_{j=1}^n\omega_j (I_{\mathscr H}-A_j)^p\right)\,,
\end{equation}
where $\Phi\in\mathbf{P}_N[\mathbb(\mathscr H), \mathbb(\mathscr K)]$, $A_j$ are contractions, $0 < p < 1$ and $\omega_j$ are real positive numbers such that $\sum_{j=1}^n\omega_j=1$.

In \cite[Theorem 2.5]{MMM}, the authors presented an equivalent form of Bellman inequality. With similar proof, the following theorem holds
\begin{theorem}\label{thBell}
The following equivalent statements hold:
\begin{enumerate}
\item[(i)] If $m, n$ are positive integers, $0 < p < 1$, $\omega_1,\cdots,\omega_n\in\mathbb R_+$ are any finite number of positive real numbers such that $\sum_{j=1}^n\omega_j=1$ and  $a_{ij}$ $(j=1,\cdots n,~i=1,\cdots,m)$ are positive real numbers such that $\sum_{i=1}^ma_{ij}^p \leq 1$ for all $j=1,\cdots n$, then
\begin{eqnarray}\label{mp3}
\sum_{j=1}^n\omega_j\left(1-\sum_{i=1}^ma_{ij}^{\frac{1}{p}}\right)^{p} \leq \left(1-\sum_{i=1}^m\Big(\sum_{j=1}^n \omega_ja_{ij}\Big)^{\frac{1}{p}}\right)^{p}\,.
\end{eqnarray}
\item[(ii)] (Classical Bellman Inequality) If $n$ is a positive integer, $0 < p < 1$ and $M_i, a_{ij}$ $(j=1,\cdots n,~i=1,\cdots,m)$ are nonnegative real numbers such that $\sum_{i=1}^ma_{ij}^{1/p} \leq M_j^{1/p}$ for all $j=1,\cdots n$, then
\begin{eqnarray}\label{mp1}
\sum_{j=1}^n\left(M_j^{\frac{1}{p}}-\sum_{i=1}^ma_{ij}^{\frac{1}{p}}\right)^{p} \le \left(\left(\sum_{j=1}^nM_j\right)^{\frac{1}{p}}-\sum_{i=1}^m\left(\sum_{j=1}^na_{ij}\right)^{\frac{1}{p}}\right)^{p}\,.
\end{eqnarray}
\end{enumerate}
\end{theorem}

Now, we state reverse of \eqref{BELL} by following result.
\begin{corollary}[A second type reverse operator Bellman inequality]\label{tdr}
Let $A_j, \Phi_j, \omega_j, j=1, \cdots, n$ be as above, $A_j$ be contractions such that $0\leq mI_{\mathscr H}\le A_j \le MI_{\mathscr H} < I_{\mathscr H}$ and $0 < p < 1$.  Then
\begin{equation}\label{eq1}
\delta I_{\mathscr K}+\sum_{j=1}^n\omega_j\Phi_j\left((I_{\mathscr H}-A_j)^{p}\right)\ge \left(\sum_{j=1}^n\omega_j\Phi_j(I_{\mathscr H}-A_j)\right)^{p} \,,
\end{equation}
where $\delta=(1-p)\left(\frac{1}{p}\frac{(1-m)^p-(1-M)^p}{M-m}\right)^{\frac{p}{p-1}}+\frac{(1-M)(1-m)^p-(1-m)(1-M)^p}{M-m}$.

In particular,
\begin{equation}\label{eq2}
\delta I_{\mathscr K}+\Phi\left(\sum_{j=1}^n\omega_j(I_{\mathscr H}-A_j)^{p}\right)\geq\left(\Phi\left(\sum_{j=1}^n(I_{\mathscr H}-A_j)\right)\right)^{p}.
\end{equation}
\end{corollary}
\begin{proof}
Note that the function $g(t)=t^r$ is operator concave on $(0,\infty)$ when $0\le r\le1$ and so is the function $f(t)=(1-t)^p$ on $(0,1)$ when $0\le p\le1$. It follows from the linearity and the normality of $\Phi_j$ that
\begin{align*}
\left(\sum_{j=1}^n\omega_j\Phi_j(I_{\mathscr H}-A_j)\right)^{p}
&=\left(\sum_{j=1}^n\omega_j\big(I_{\mathscr K}-\Phi_j(A_j)\big)\right)^{p}\\
&=\left(I_{\mathscr K}-\sum_{j=1}^n\omega_j\Phi_j(A_j)\right)^{p}\\
&=f\left(\sum_{j=1}^n\omega_j\Phi_j(A_j)\right)\\
&\le \sum_{j=1}^n\omega_j\Phi_j(f(A_j))+\beta_fI_{\mathscr K}\hspace{2cm}(\mathrm{by}~ \eqref{JR})\\
&=\sum_{j=1}^n\omega_j\Phi_j\left((I_{\mathscr H}-A_j)^{p}\right)+\beta_fI_{\mathscr K}\,.
\end{align*}
Since $f(t)=(1-t)^p$ is a differentiable function on $(0,1)$, the function $h(t):={(1-t)^p-\frac{(1-M)^p-(1-m)^p}{M-m} t-\frac{M(1-m)^p-m(1-M)^p}{M-m}}~~ (m\leq t\leq M)$  attained its maximum value in $t_0=1-\left(\frac{1}{p}\frac{(1-m)^p-(1-M)^p}{M-m}\right)^{\frac{1}{p-1}}$ which is equal to
{\footnotesize
$$
\delta=\max_{m \le t \le M} h(t)=(1-p)\left(\frac{1}{p}\frac{(1-m)^p-(1-M)^p}{M-m}\right)^{\frac{p}{p-1}}+\frac{(1-M)(1-m)^p-(1-m)(1-M)^p}{M-m}\,.
$$
}
For inequality \eqref{eq2}, it is enough to put $\Phi_j=\Phi\quad(j=1,\cdots,n)$ and use the concavity of $f$ and linearity of $\Phi$.
\end{proof}

\begin{corollary}\label{cr1}
If $m, n$ are positive integers, $0 < p < 1$, $\omega_1,\cdots,\omega_n\in\mathbb R_+$ are any finite number of positive real numbers such that $\sum_{j=1}^n\omega_j=1$ and  $a_{ij}$ $(j=1,\cdots n,~i=1,\cdots,m)$ are positive real numbers such that $1\ge\sum_{i=1}^ma_{ij}^{1/p}\quad(j=1,\cdots n)$, then
\begin{equation}\label{eq3}
(1-p)p^{\frac{p}{1-p}}+\sum_{j=1}^n\omega_j\left(1-\sum_{i=1}^ma_{ij}^{\frac{1}{p}}\right)^{p} \geq \left(1-\sum_{i=1}^m\sum_{j=1}^n\omega_ja_{ij}^{\frac{1}{p}}\right)^{p}\,.
\end{equation}
\end{corollary}
\begin{proof}
Let $m, n$ be positive integers, $0 < p <1$, $0 \le \omega_j \le 1$ and  $a_{ij}$ $(1 \le j \le n,~1 \le i \le m)$ are positive real numbers such that $1\ge\sum_{i=1}^ma_{ij}^{1/p}~(j=1,\cdots n)$. Set $A_j=\left[\begin{array}{cc}\sum_{i=1}^ma_{ij}^{1/p}&0\\0&1\end{array}\right]\in \mathcal{M}_2(\mathbb{C})~(j=1,\cdots, n)$. Then
\begin{align*}
\sum_{j=1}^n\omega_j(I_2-A_j)^{p}&=\sum_{j=1}^n\omega_j\left(\left[\begin{array}{cc}1&0\\0&1\end{array}\right]-
\left[\begin{array}{cc}\sum_{i=1}^ma_{ij}^{\frac{1}{p}}&0\\0&1\end{array}\right]\right)^{p}\\
&=\sum_{j=1}^n\omega_j\left[\begin{array}{cc}\left(1-\sum_{i=1}^ma_{ij}^{\frac{1}{p}}\right)^{p}&0\\0&0\end{array}\right]\\
&=\left[\begin{array}{cc}\sum_{j=1}^n\omega_j\left(1-\sum_{i=1}^ma_{ij}^{\frac{1}{p}}\right)^{p}&0\\0&0\end{array}\right]\,.
\end{align*}
and
\begin{align*}
\left(\sum_{j=1}^n\omega_j(I_2-A_j)\right)^{p}&=\left(\sum_{j=1}^n\omega_j\left(\left[\begin{array}{cc}1&0\\0&1\end{array}\right]-
\left[\begin{array}{cc}\sum_{i=1}^ma_{ij}^{\frac{1}{p}}&0\\0&1
\end{array}\right]\right)\right)^{p}\\
&=\left[\begin{array}{cc}\sum_{j=1}^n\omega_j(1-\sum_{i=1}^ma_{ij}^{\frac{1}{p}})&0\\0&0\end{array}\right]^{p}\\
&=\left[\begin{array}{cc}\left(1-\sum_{j=1}^n\sum_{i=1}^m\omega_ja_{ij}^{\frac{1}{p}}\right)^{p}&0\\0&0\end{array}\right]\,.
\end{align*}
It follows from \eqref{eq2} with the identity map $\Phi$ that
\begin{align*}
\delta I_2+\left[\begin{array}{cc}\sum_{j=1}^n\omega_j\left(1-\sum_{i=1}^ma_{ij}^{\frac{1}{p}}\right)^{p}&0\\0&0\end{array}\right]\geq
\left[\begin{array}{cc}\left(1-\sum_{j=1}^n\sum_{i=1}^m\omega_ja_{ij}^{\frac{1}{p}}\right)^{p}&0\\0&0\end{array}\right] \,,
\end{align*}
where $\delta=(1-p)p^{\frac{p}{1-p}}$ by taking $m=0$ and $M=1$ in \eqref{eq1}, which gives \eqref{eq3}.\\
\end{proof}
\begin{corollary}\label{cr2}
Let $\Phi\in \mathbf{P}_N[\mathcal{B}(\mathscr{H}),\mathcal{B}(\mathscr{K})]$, $0< mI_{\mathscr H} \le A_j \le MI_{\mathscr H}$ be positive operators and $0 \leq \omega_j \leq 1~(j=1,\cdots,n)$ such that $\sum_{j=1}^n\omega_j=1$. Then
$$
\log\left[\frac{1}{e}\left(\frac{M^m}{m^M}\right)^{\frac{1}{M-m}}L(m,M)\right]+\Phi\left(\sum_{j=1}^n\omega_j\log A_j\right)\geq\log\left(\sum_{j=1}^n\omega_j\Phi(A_j)\right)$$
where $L(a,b)=\begin{cases}\frac{b-a}{\log b-\log a} & ;a\neq b\\ a & ; a=b\end{cases}$ is the Logarithmic mean of positive real numbers $a$ and $b$.
\end{corollary}
\begin{proof}
Put $f(t)=\log t$ and $\Phi_j=\Phi$ in \eqref{JR}.
\end{proof}
\section{Some refinements of the Bellman operator inequality}
In this section, we present some refinements of the operator Bellman  inequality by using some ideas of \cite{pop}. First we need the following Lemmas.
\begin{lemma}
Let $A,B, A_j,B_j,\,\,(1\leq j\leq n)$ be positive operators such that $\sum_{j=1}^nA_j\leq A$, $\sum_{j=1}^nB_j\leq B$ and let $\sigma_f$ be an operator mean with the representing function $f$. Then
\begin{align}\label{mor243}
\left(A-\sum_{j=1}^nA_j\right)\sigma_{f}\left(B-\sum_{j=1}^nB_j\right)
\leq(A\sigma_f B)-\sum_{j=1}^n\left(A_j\sigma_fB_j\right).
\end{align}
\end{lemma}
\begin{proof}
The subadditivity of operator mean says that \cite[Theorem 5.7]{abc}
\begin{align}\label{f-h23}
\sum_{j=1}^{n+1}\left(X_j\sigma_fY_j\right)\leq
\left(\sum_{j=1}^{n+1}X_j\right)\sigma_{f}\left(\sum_{j=1}^{n+1}Y_j\right),
\end{align}
where $X_j,Y_j,\,\,(1\leq j\leq n+1)$ are positive operators. If we put $X_j=A_j, Y_j=B_j\,\,(1\leq j\leq n)$, $X_{n+1}=A-\sum_{j=1}^nA_j$ and $Y_{n+1}=A-\sum_{j=1}^nB_j$ in inequality \eqref{f-h23}, then we reach
\begin{align*}
\sum_{j=1}^{n}\left(A_j\sigma_fB_j\right)+\left(A-\sum_{j=1}^nA_j\right)\sigma_f\left(B-\sum_{j=1}^nB_j\right)
\leq A\sigma_{f}B.
\end{align*}
Therefore
\begin{align*}
\left(A-\sum_{j=1}^nA_j\right)\sigma_f\left(B-\sum_{j=1}^nB_j\right)
\leq \left(A\sigma_{f}B\right)-\sum_{j=1}^{n}\left(A_j\sigma_fB_j\right).
\end{align*}
\end{proof}
\begin{lemma}\label{morasaee33}\cite[Lemma 2.1]{MMM1}
Let $A,B\in\mathbb{B}(\mathscr{H})$ be positive operators such that $A$ is contraction, $h$ is a nonnegative
operator monotone function on $[0,+\infty)$ and $\sigma_f$ be an operator mean with the representing function $f$. Then
\begin{align*}
A\sigma_{hof}B\leq h(A\sigma_{f}B).
\end{align*}
\end{lemma}
In the next theorem, we show a refinement of \eqref{mor2}.
\begin{theorem}\label{54398}
Let $A_j,B_j,\,\,(1\leq j\leq n)$ be positive operators such that $\sum_{j=1}^nA_j\leq I_{\mathscr H}$, $\sum_{j=1}^nB_j\leq I_{\mathscr H}$, $\sigma_f$ be an operator mean with the representing function $f$ and $p\in[0,1]$. Then
{\footnotesize\begin{align*}
\left(I_{\mathscr H}-\sum_{j=1}^nA_j\right)\sigma_{f^p}\left(I_{\mathscr H}-\sum_{j=1}^nB_j\right)
&\leq\left(\left(I_{\mathscr H}-\sum_{j=1}^kA_j\right)\sigma_f \left(I_{\mathscr H}-\sum_{j=1}^kB_j\right)-\sum_{j=k+1}^n\left(A_j\sigma_fB_j\right)\right)^{p}
\\&\leq\left(I_{\mathscr H}-\sum_{j=1}^n\left(A_j\sigma_fB_j\right)\right)^{p},
\end{align*}}
in which $k=1,2,\cdots,n-1$.
\end{theorem}
\begin{proof}
For $k=1,2,\cdots,n-1$ we have
{\footnotesize\begin{align}\label{hohoh}
\left(I_{\mathscr H}-\sum_{j=1}^nA_j\right)&\sigma_{f}\left(I_{\mathscr H}-\sum_{j=1}^nB_j\right)\nonumber
\\&=\left(\Big(I_{\mathscr H}-\sum_{j=1}^kA_j\Big)-\sum_{j=k+1}^nA_j\right)\sigma_{f}\left(\Big(I_{\mathscr H}-\sum_{j=1}^kB_j\Big)-\sum_{j=k+1}^nB_j\right)\nonumber
\\&\leq\left(I_{\mathscr H}-\sum_{j=1}^kA_j\right)\sigma_{f}\left(I_{\mathscr H}-\sum_{j=1}^kB_j\right)-\sum_{j=k+1}^n(A_j\sigma_{f}B_j)\,\,(\textrm{by \eqref{mor243}})\nonumber
\\&\leq\left(\left(I_{\mathscr H}\sigma_f I_{\mathscr H}\right)- \sum_{j=1}^k(A_j\sigma_fB_j)\right)-\sum_{j=k+1}^n\left(A_j\sigma_fB_j\right)\quad(\textrm{by \eqref{mor243}})\nonumber
\\&=I_{\mathscr H}-\sum_{j=1}^n\left(A_j\sigma_fB_j\right),
\end{align}}
whence for the spacial case  $g(t)=t^{p}\,\,(p\in[0,1])$  we have
\begin{align*}
\left(I_{\mathscr H}-\sum_{j=1}^nA_j\right)&\sigma_{f^p}\left(I_{\mathscr H}-\sum_{j=1}^nB_j\right)
\\&\leq\left(\left(I_{\mathscr H}-\sum_{j=1}^nA_j\right)\sigma_{f}\left(I_{\mathscr H}-\sum_{j=1}^nB_j\right)\right)^{p}(\textrm{by Lemma \ref{morasaee33}})
\\&\leq\left(\left(I_{\mathscr H}-\sum_{j=1}^kA_j\right)\sigma_{f}\left(I_{\mathscr H}-\sum_{j=1}^kB_j\right)-\sum_{j=k+1}^n(A_j\sigma_{f}B_j)\right)^{p}
\\&=\left(I_{\mathscr H}-\sum_{j=1}^n\left(A_j\sigma_fB_j\right)\right)^{p}.
\end{align*}
\end{proof}
Using the same idea as in the proof of Theorem \ref{54398} we improve  inequality \eqref{mor2} in the next theorem.
\begin{theorem}
Let $A_j,B_j,\,\,(1\leq j\leq n)$ be positive operators such that $\sum_{j=1}^nA_j\leq I_{\mathscr H}$, $\sum_{j=1}^nB_j\leq I_{\mathscr H}$, $\sigma_f$  be an operator mean with the representing function $f$ and $p\in[0,1]$. Then
{\footnotesize\begin{align*}
\left(I_{\mathscr H}-\sum_{j=1}^nA_j\right)&\sigma_{f^p}\left(I_{\mathscr H}-\sum_{j=1}^nB_j\right)
\\&\leq\left(\left(I_{\mathscr H}-\sum_{j=1}^nt_jA_j\right)\sigma_f \left(I_{\mathscr H}-\sum_{j=1}^nt_jB_j\right)-\sum_{j=1}^n(1-t_j)\left(A_j\sigma_fB_j\right)\right)^{p}
\\&\leq\left(I_{\mathscr H}-\sum_{j=1}^n\left(A_j\sigma_fB_j\right)\right)^{p},
\end{align*}}
where  $t_j\in[0,1]\,\,(j=1,\cdots,n)$.
\end{theorem}
\begin{proof}
Let  $t_j\in[0,1]\,\,(j=1,\cdots,n)$. It follows from  $I_{\mathscr H}-\sum_{j=1}^nt_jA_j\geq\sum_{j=1}^n(1-t_j)A_j$, $I_{\mathscr H}-\sum_{j=1}^nt_jB_j\geq\sum_{j=1}^n(1-t_j)B_j$ and inequality \eqref{mor243} that
\begin{align}\label{user}
&\left(I_{\mathscr H}-\sum_{j=1}^nA_j\right)\sigma_{f}\left(I_{\mathscr H}-\sum_{j=1}^nB_j\right)\nonumber
\\&=\left(\Big(I_{\mathscr H}-\sum_{j=1}^nt_jA_j\Big)-\sum_{j=1}^n(1-t_j)A_j\right)\sigma_{f}\left(\Big(I_{\mathscr H}-\sum_{j=1}^nt_jB_j\Big)-\sum_{j=1}^n(1-t_j)B_j\right)\nonumber
\\&\leq\left(\Big(I_{\mathscr H}-\sum_{j=1}^nt_jA_j\Big)\sigma_f \Big(I_{\mathscr H}-\sum_{j=1}^nt_jB_j\Big)-\sum_{j=1}^n\left((1-t_j)A_j\sigma_f(1-t_j)B_j\right)\right)~{\footnotesize(\textrm{by \eqref{mor243}})}\nonumber
\\&=\left(\Big(I_{\mathscr H}-\sum_{j=1}^nt_jA_j\Big)\sigma_f \Big(I_{\mathscr H}-\sum_{j=1}^nt_jB_j\Big)-\sum_{j=1}^n(1-t_j)\left(A_j\sigma_fB_j\right)\right)\nonumber\\&\qquad\qquad(\textrm{by property (iii) of operator means})\nonumber
\\&\leq\left((I_{\mathscr H}\sigma_fI_{\mathscr H})-\sum_{j=1}^n(t_jA_j\sigma_ft_jB_j)-\sum_{j=1}^n(1-t_j)\left(A_j\sigma_fB_j\right)\right)\,\,(\textrm{by \eqref{mor243}})\nonumber
\\&=\left(I_{\mathscr H}-\sum_{j=1}^nt_j(A_j\sigma_fB_j)-\sum_{j=1}^n(1-t_j)\left(A_j\sigma_fB_j\right)\right)\nonumber
\\&\qquad\qquad(\textrm{by property (iii) of operator means})\nonumber
\\&=\left(I_{\mathscr H}-\sum_{j=1}^n\left(A_j\sigma_fB_j\right)\right).
\end{align}
Using Lemma \ref{morasaee33}, the operator monotone function $g(t)=t^{p}\,\,(p\in[0,1])$ and inequality \eqref{user} we get the desired result.
\end{proof}

\bigskip
\textbf{Acknowledgement.} The authors would like to sincerely thank the anonymous  referee for some useful comments and suggestions. The first author would like to thank the Tusi Mathematical Research Group (TMRG).
\bigskip
\bigskip
\bibliographystyle{amsplain}

\end{document}